\documentclass{amsart}
\usepackage{amsmath,amsthm,amssymb}
\usepackage{graphicx}

\newtheorem{theorem}{Theorem}[section]
\newtheorem{proposition}[theorem]{Proposition}
\newtheorem{lemma}[theorem]{Lemma}
\newtheorem{corollary}[theorem]{Corollary}

\theoremstyle{definition}

\newtheorem{definition}[theorem]{Definition}
\newtheorem{remark}[theorem]{Remark}

\newtheorem{question}[theorem]{Question}

\theoremstyle{theorem}

\title[Contact 5-manifolds and smooth structures on Stein 4-manifolds]{Contact 5-manifolds and smooth structures\\on Stein 4-manifolds}

\author[Kouichi Yasui]{Kouichi Yasui}
\address{Department~of~Mathematics, Graduate School~of~Science, Hiroshima~University, 1-3-1 Kagamiyama, Higashi-Hiroshima, 739-8526, Japan}
\email{kyasui@hiroshima-u.ac.jp}
\thanks{The author was partially supported by JSPS KAKENHI Grant Number 16K17593.}
\date{April 12, 2016}
\keywords{4-manifold; Stein structure; contact structure; open book; minimal genus function}
\subjclass[2010]{Primary~57R55, Secondary~57R65, 57R17}

\begin{document}
\maketitle

\begin{abstract}We show that, under a certain condition, contact 5-manifolds can `coarsely' distinguish smooth structures on compact Stein 4-manifolds via contact open books. We also give a simple sufficient condition for an infinite family of Stein 4-manifolds to have an infinite subfamily of pairwise non-diffeomorphic Stein 4-manifolds. The proofs rely on the adjunction inequality. We remark that there are various examples of infinitely many pairwise exotic Stein 4-manifolds whose smooth structures can be distinguished by these results.
\end{abstract}

\section{Introduction}\label{sec:intro}
We consider abstract open books on closed oriented smooth 5-manifolds whose pages are compact Stein 4-manifolds and whose monodromies are the identity maps, which are a symplectomorphism common to all Stein 4-manifolds. It is known that the Stein structure on the page of an open book induces a contact structure on the smooth 5-manifold (\cite{Gi}, cf.\ \cite{DGV}), and an open book equipped with this contact structure is called a \textit{contact open book}. Van Koert \cite{vK} and Ding-Geiges-van Koert \cite{DGV} studied such contact 5-manifolds and then classified contact 5-manifolds admitting subcritical Stein fillings without 1-handles in \cite{DGV}. 

From the viewpoint of 4-dimensional topology, it is natural to ask how the smooth structures on the pages of contact open books and the supporting contact structures are related to each other. 
Ozbagci-van Koert~\cite{OV} recently showed that infinitely many contact open books with pairwise exotic (i.e.\ homeomorphic but non-diffeomorphic) pages (obtained in \cite{AY6}) can support pairwise non-contactomorphic contact structures on fixed smooth 5-manifolds. Akbulut and the author~\cite{AY7} extended this result to infinitely many smooth 5-manifolds, using exotic 4-manifolds obtained in \cite{Y6}. By contrast, Akbulut and the author~\cite{AY7} also proved that infinitely many contact open books with pairwise exotic pages can support the same contact 5-manifold, using exotic 4-manifolds obtained in \cite{AY6} and \cite{Y6}. 

In this paper we show that, under a certain condition, closed contact 5-manifolds can `coarsely' distinguish smooth structures on pages of their supporting contact open books. We note that such contact 5-manifolds are not invariants of smooth structures on pages (\cite{DGV}. cf.\ Proposition~\ref{prop:strengthen}). For the definition of subcritical Stein fillings without 1-handles, see Section~\ref{sec:Stein}. 

\begin{theorem}\label{intro:thm:main}
Let $\{X_i\mid i\in \mathbb{N}\}$ be an infinite family of compact Stein 4-manifolds such that, for each $i$, the contact 5-manifold $(W_i,\xi_i)$ supported by the contact open book $(X_i, \textnormal{id})$ admits a subcritical Stein filling without 1-handles. If the members of $\{(W_i,\xi_i)\mid i\in \mathbb{N}\}$ are pairwise non-contactomorphic, then at least infinitely many members of $\{X_i\mid i\in \mathbb{N}\}$ are pairwise non-diffeomorphic.
%Let $\{(W_i,\xi_i)\mid i\in \mathbb{N}\}$ be an infinite family of contact 5-manifolds admitting subcritical Stein fillings without 1-handles, and for each $i\in \mathbb{N}$, let $X_i$ be any compact Stein 4-manifold supporting $(W_i,\xi_i)$ by the contact open book $(X_i,id)$. If these contact 5-manifolds are pairwise non-contactomorphic, then at least infinitely many members of $\{X_i\mid i\in \mathbb{N}\}$ are pairwise non-diffeomorphic.
\end{theorem}
\begin{remark}$(1)$ If the page of a contact open book (with the identity monodromy) is a 4-dimensional Stein handlebody without 1-handles, then the contact 5-manifold supported by the contact open book admits a subcritical Stein filling without 1-handles (see Remark~\ref{rem:subcritical}). For more general sufficient condition, see Lemma~\ref{lem:AC}.
\\
$(2)$ There are various examples of infinitely many pairwise exotic 4-manifolds (obtained in \cite{Y6}) which are distinguished by this theorem (see Theorem~\ref{thm:exotic}).%\smallskip
\\
$(3)$ The condition `at least infinitely many members' in this theorem cannot be strengthened to `all members' without additional constraints (see Proposition~\ref{prop:strengthen}).%\smallskip
\\
$(4)$ The converse of this theorem does not hold. In fact, infinitely many pairwise exotic Stein handlebodies without 1-handles can be the pages of the same contact 5-manifold (Theorem~1.2 in \cite{AY7}). 
\end{remark}

The proof of this theorem relies on the adjunction inequality --- a result from Seiberg-Witten theory --- and the classification of contact 5-manifolds admitting subcritical Stein fillings without 1-handles in \cite{DGV}.

We moreover give a simple and practical sufficient condition for an infinite family of Stein 4-manifolds to have an infinite subfamily of pairwise non-diffeomorphic Stein 4-manifolds. For our definition of the divisibility, see Subsection~\ref{subsec:distinguish}. 

\begin{theorem}\label{thm:Stein:exotic}For a given infinite family of compact Stein 4-manifolds, if the divisibilities of the first Chern classes of the members are pairwise distinct, then at least infinitely many members are pairwise non-diffeomorphic. 
\end{theorem}
\begin{remark}(1) This theorem works for various families of pairwise exotic Stein 4-manifolds. For example, for any finitely presented group $G$, the algorithm obtained in \cite{Y6} can produce such infinite families with $\pi_1\cong G$. Moreover, the algorithm can realize a vast class of integral symmetric bilinear forms as intersection forms of such infinite families. However, the algorithm also produces various infinite families of exotic Stein 4-manifolds that do not satisfy the assumption of this theorem. 
%\smallskip
\\
(2) The condition `at least infinitely many members' in this theorem cannot be strengthened to `all members' without additional constraints (see the proof of Proposition~\ref{prop:strengthen}). 
\end{remark}

To distinguish smooth structures, we introduce the \textit{intersection genus}, which is a simple genus invariant (i.e.\ a diffeomorphism invariant determined by the minimal genus function on a 4-manifold). For other genus invariants, we refer to the relative genus function in \cite{Y5} (cf.\ \cite{Y6}) and the genus rank function in \cite{G_GT}. See also Subsection~\ref{subsec:genus} for the background of these genus invariants. The intersection genus is simpler than these genus invariants but more useful for our purpose. Estimating intersection genera of Stein 4-manifolds by the adjunction inequality, we prove the above theorem. 

We also show the finiteness of the number of the first Chern classes of Stein structures on a given compact smooth 4-manifold with boundary (Proposition~\ref{prop:finiteness:exotic}). This result yields an alternative proof of Theorem~\ref{thm:Stein:exotic}, though this proof does not tell how to distinguish smooth structures on given two Stein 4-manifolds unlike the one relying on the intersection genus. 
%%%%%%%%%%%%%%%%
\section{Compact Stein manifolds and the adjunction inequality}\label{sec:Stein}
In this section we briefly recall the basics of compact Stein manifolds. For the details, see \cite{OS1}, \cite{DGV} and \cite{CE}. We begin with a few definitions on Stein manifolds of arbitrary dimension. A compact Stein manifold is a compact complex manifold with boundary admitting a strictly plurisubharmonic function that is constant on the boundary and has no critical points on the boundary. If a compact Stein manifold $W$ of real dimension $2n$ admits a strictly plurisubharmonic function that is a Morse function without critical points of index $n$, then $W$ is called subcritical. If the Morse function on $W$ furthermore has no critical points of index one, then $W$ is called a subcritical Stein manifold without 1-handles.  A compact Stein manifold $W$ is called a Stein filling of a contact manifold $(M,\xi)$ if $(M,\xi)$ is contactomorphic to the contact structure on the boundary $\partial W$ given by the complex tangencies. 

Now we restrict our attention to the case of real dimension four. We call a compact connected oriented 4-dimensional handlebody a \textit{Stein handlebody}, if it consists of one 0-handle and 1- and 2-handles, and its 2-handles are attached along a Legendrian link in the standard tight contact structure on the boundary $\#_n S^1\times S^2$ of the subhandlebody consisting of $0$ and 1-handles, where the framing of each 2-handle is one less than the framing induced from the contact plane (i.e.\ contact framing). According to \cite{E1} (cf.\ \cite{G1}, \cite{GS}), any 4-dimensional Stein handlebody admits a Stein structure extending that of the 0-handle $D^4$, and conversely any compact Stein 4-manifold (i.e.\ compact Stein manifold of real dimension four) is diffeomorphic to a Stein handlebody. 

The following \textit{adjunction inequality} gives a strong constraints on the genera of smoothly embedded surfaces representing a given second homology class. This inequality, which is a result from Seiberg-Witten theory, plays a key role in this paper.

\begin{theorem}[\cite{LM1}, \cite{AM}, \cite{LM2}]Let $X$ be a compact Stein 4-manifold, and let $[\Sigma]$ be a non-zero class of $H_2(X;\mathbb{Z})$ represented by a smoothly embedded closed connected oriented surface $\Sigma$ of genus $g\geq 0$ in $X$. Then the following adjunction inequality holds. 
\begin{equation*}
\left| \langle c_1(Z), [\Sigma] \rangle \right|+ [\Sigma]\cdot [\Sigma]\leq 2g-2. 
\end{equation*}
\end{theorem}
This well-known result follows from the adjunction inequality for closed 4-manifold (\cite{FS3}, \cite{KM1}, \cite{MST}, \cite{OzSz_ad}), since any compact Stein 4-manifold can be embedded into a closed minimal complex surface of general type with $b_2^+>1$ (\cite{LM1}). Note that the adjunction inequality for Stein 4-manifolds holds even in the genus zero case (cf.\ \cite{GS}, \cite{OS1}, \cite{AY5}), unlike the version for general closed 4-manifolds. 

%%%%%%%%%%%%%%%%%%%%%%%%%%%
\section{Contact open books on 5-manifolds}In this section, we briefly review the basics of contact 5-manifolds supported by contact open books. For the details, see \cite{DGV}. 
 
Let $(X,\psi)$ be a 5-dimensional abstract open book whose page is a compact Stein 4-manifold $X$ and whose monodromy is a symplectomorphism $\psi$ equal to the identity near the boundary. It is known that the Stein structure on the page $X$ and the monodromy $\psi$ induces a contact structure on the closed smooth 5-manifold given by $(X,\psi)$ (\cite{Gi}.\ cf.\ \cite{DGV}). The open book $(X,\psi)$ equipped with this contact structure is called the \textit{contact open book}. We say that a contact 5-manifold $(W,\xi)$ is supported by a contact open book $(X,\psi)$ if $(W,\xi)$ is contactomorphic to the contact 5-manifold given by the contact open book $(X,\psi)$. 

In the rest, we consider the case where the monodromy is the identity map, which is a symplectomorphism common to all Stein 4-manifolds. Relying on a result of Cieliebak~\cite{C}, Ding-Geiges-van Koert observed a characterization of  contact 5-manifolds supported by contact open books. 
\begin{theorem}[\cite{C}, Proposition~3.1 in \cite{DGV}]\label{thm:C}A closed contact 5-manifold admits a subcritical Stein filling if and only if the contact 5-manifold is supported by a contact open book whose page is a compact Stein 4-manifold and whose monodromy is the identity map. 
\end{theorem}
\begin{remark}\label{rem:subcritical}A compact Stein 4-manifold $X$ gives a subcritical Stein 6-manifold $X\times D^2$, and its boundary contact structure is supported by the contact open book $(X,\textnormal{id})$ . Furthermore, if $X$ is a Stein handlebody without 1-handles, then $X\times D^2$ is a subcritical Stein 6-manifold without 1-handles. For these facts, see the proof of Proposition~3.1 in \cite{DGV}. Conversely, a subcritical Stein 6-manifold $M$ is deformation equivalent to the Stein 6-manifold $X\times D^2$ for some 4-dimensional Stein handlebody $X$. If $M$ has no 1-handles, then we may assume that $X$ has no 1-handles. For the proof of these facts, see the proof of Theorem~1.1 in \cite{C} (cf.\ Section~14.4 in \cite{CE}). 
\end{remark}

Ding-Geiges-van Koert classified contact 5-manifolds admitting subcritical Stein fillings without 1-handles. We note that, due to the above remark, a contact 5-manifold admits a subcritical Stein filling without 1-handles if and only if it is supported by a contact open book whose page is a Stein handlebody without 1-handles and whose monodromy is the identity map.

\begin{theorem}[Ding-Geiges-van Koert, Theorem 4.8 in \cite{DGV}]\label{thm:DGV}Let $(W_1,\xi_1)$ and $(W_2, \xi_2)$ be two closed contact 5-manifolds admitting subcritical Stein fillings without 1-handles. If there exists an isomorphism $H^2(W_1;\mathbb{Z})\to H^2(W_2;\mathbb{Z})$ that sends $c_1(\xi_1)$ to $c_1(\xi_2)$, then $(W_1,\xi_1)$ and $(W_2, \xi_2)$ are contactomorphic to each other. 
\end{theorem}

The lemma below was observed by Ding-Geiges-van Koert (see the proof of Proposition 4.5 in \cite{DGV}). 

\begin{lemma}[\cite{DGV}]\label{lem:iso}
For a compact Stein 4-manifold $X$ and a closed contact 5-manifold $(W,\xi)$ supported by the contact open book $(X, \textnormal{id})$, there exists an isomorphism $H^2(X;\mathbb{Z})\to H^2(W;\mathbb{Z})$ that maps $c_1(X)$ to $c_1(\xi)$. 
\end{lemma}

%\begin{remark}A closed contact 5-manifold $(W,\xi)$ admitting a subcritical Stein filling without 1-handles is diffeomorphic to either $\#_nS^2\times S^3$ or $\#_nS^2\widetilde{\times} S^3$, where $n=b_2(W)$ and $S^2\widetilde{\times} S^3$ is the total space of the non-trivial $S^3$-bundle over $S^2$ (\cite{DGV}. cf.\ Remark~3.6.(2) in \cite{AY7}). Theorems~\ref{thm:C} and \ref{thm:DGV} and Lemma~\ref{lem:iso} implies that  \end{remark}

We recall the following terminology in \cite{AY7}, which is useful for describing contact 5-manifolds. 

\begin{definition}
For a 4-dimensional Stein handlebody $X$ without 1-handles, we define the rotation divisor of $X$, denoted by $r(X)$, as the greatest common divisor of the rotation numbers of the attaching circles of the 2-handles of $X$. If $b_2(X)=0$, or if all the rotation numbers of the attaching circles are $0$, then we define $r(X)$ by $r(X)=0$. 
Note that each attaching circle is a Legendrian knot in the standard tight contact structure on $S^3$. 
\end{definition}

Let $S^2\widetilde{\times} S^3$ denote the total space of the non-trivial $S^3$-bundle over $S^2$. The following proposition was observed by Akbulut and the author. This easily follows from Theorems~\ref{thm:C} and \ref{thm:DGV}, Lemma~\ref{lem:iso} and Gompf's first Chern class formula for Stein handlebodies~\cite{G1}. 

\begin{proposition}[\cite{AY7}]\label{prop:rotation} For two 4-dimensional Stein handlebodies $X_1, X_2$ with $b_2=n$ and without 1-handles, the contact 5-manifolds $(W_1, \xi_1)$ and $(W_2,\xi_2)$ supported by the contact open books $(X_1,\textnormal{id})$ and $(X_2,\textnormal{id})$ are contactomorphic to each other if and only if $r(X_1)=r(X_2)$. Furthermore, $W_1$ is diffeomorphic to $\#_nS^2\times S^3$  $($resp.\ $\#_n S^2\tilde{\times} S^3$$)$, if $r(X_1)$ is even $($resp.\  odd$)$.
\end{proposition}

This proposition leads to the following definition. 
\begin{definition}[\cite{AY7}]\label{def:classification}
For non-negative integers $n$ and $r$, we define $(S_{n,r}, \zeta_{n,r})$ as the closed contact 5-manifold supported by a contact open book $(X,\textnormal{id})$, where $X$ is a 4-dimensional Stein handlebody without 1-handles satisfying $b_2(X)=n$ and $r(X)=r$. %We define $(S_{0,0}, \zeta_{0,0})$ as the closed contact 5-manifold supported by
\end{definition}

\begin{remark}\label{rem:classification}
$(1)$ By Proposition~\ref{prop:rotation}, $(S_{n,r},\zeta_{n,r})$ and $(S_{n',r'},\zeta_{n', r'})$ are contactomorphic to each other if and only if $n=n'$ and $r=r'$. Furthermore, $S_{r,n}$ is diffeomorphic to $\#_nS^2\times S^3$ $($resp.\ $\#_n S^2\tilde{\times} S^3$$)$, if $r$ is even $($resp.\  odd$)$.\\
$(2)$ By Theorem~\ref{thm:C} and Remark~\ref{rem:subcritical}, any closed contact 5-manifold admitting a subcritical Stein filling without 1-handles is contactomorphic to some $(S_{n,r}, \zeta_{n,r})$. 
\end{remark}

Here we observe that the contact 5-manifold supported by a contact open book can admit a subcritical Stein filling without 1-handles, even if the page has 1-handles. Let us recall that a handlebody determines a presentation of its fundamental group: 1- and 2-handles corresponds to the generators and relators, respectively (cf.\ \cite{GS}).

\begin{lemma}[cf.\ Ding-Geiges-van Koert~\cite{DGV}]\label{lem:AC}Let $X$ be a 4-dimensional Stein handlebody. If all the generators of the presentation of $\pi_1(X)$ given by the handlebody structure are removed by Andrews-Curtis moves, then the contact 5-manifold supported by the contact open book $(X, \textnormal{id})$ admits a subcritical Stein filling without 1-handles. 
\end{lemma}
For the definition of Andrews-Curtis moves, see \cite{DGV}. The proof of this lemma is straightforward from Section 6 in \cite{DGV}: one can alter the page $X$ so that it has no 1-handles, preserving the supporting contact 5-manifold. It seems still unknown whether there exists a finite presentation of the trivial group that do not satisfy the assumption of the above lemma. 

\begin{question}\label{question:Andrews}Can every finite presentation of the trivial group be changed into a presentation having no generators by Andrews-Curtis moves?
\end{question}

According to the Andrews-Curtis conjecture~\cite{AC}, this question is affirmative for any balanced presentation of the trivial group, but there are potential counterexamples to the conjecture (cf.\ \cite{GS}). 
%%%%%%%%%%%%%%%%%%%%%%
%%%%%%%%%%%%%%%%%%%%%%%%%
\section{Smooth structures on Stein 4-manifolds}
In this section, we first introduce the intersection genus, which is a simple but effective genus invariant of smooth 4-manifolds. Using this invariant, we prove Theorems~\ref{thm:Stein:exotic} and \ref{intro:thm:main}. 
We also show the finiteness of the number of the first Chern classes of Stein structures on a given compact smooth 4-manifold (Proposition~\ref{prop:finiteness:exotic}). 

\subsection{Intersection genera of 4-manifolds}\label{subsec:genus}
For a smooth 4-manifold $Z$, the \textit{minimal genus function} $g_Z: H_2(Z;\mathbb{Z})\to \mathbb{Z}$ (cf.\ \cite{GS}) is defined by 
\begin{equation*}
g_Z(\alpha)=\min\{g\mid \text{$\alpha$ is represented by a smoothly embedded surface of genus $g$.}\}
\end{equation*}
Minimal genus functions have useful informations on smooth structures, but it is hard to determine the values of the functions. Indeed the functions have been determined for very few 4-manifolds (e.g.\ $\mathbb{CP}^2$). Also, it is generally difficult to distinguish smooth structures by the functions, since there are many ways to identify second homology groups of two distinct 4-manifolds. 

To avoid these difficulties, the author~\cite{Y5} (cf.\ \cite{Y6}) introduced the relative genus function, which is a genus invariant (i.e.\ a diffeomorphism invariant determined by the minimal genus function on a 4-manifold). Subsequently, Gompf~\cite{G_GT} introduced a different genus invariant called the genus rank function. Here we introduce yet another genus invariant, which is simpler than these invariants but more useful for proving Theorem~\ref{thm:Stein:exotic}.  

Let $Z$ be an oriented smooth (possibly non-closed) 4-manifold with $0<b_2<\infty$. For simplicity we assume $H_2(Z;\mathbb{Z})$ has no torsion, though we can similarly define the invariant when $H_2(Z;\mathbb{Z})$ has torsion. We put $n=b_2(Z)$. 
\begin{definition}For an ordered basis $\mathbf{v}=\{v_1,v_2,\dots,v_n\}$ of $H_2(Z;\mathbb{Z})$, we define a non-negative integer $G_Z(\mathbf{v})$ by 
\begin{equation*}
G_Z(\mathbf{v})=\max \{g_Z(v_i)\mid 1\leq i\leq n\}.
\end{equation*}
For an intersection matrix $Q$ of $Z$ (i.e.\ a matrix representing the intersection form of $Z$), we then define a non-negative integer $G_{Z,Q}$ by 
\begin{align*}
G_{Z,Q}=\min \{G_Z(\mathbf{v})\mid \text{$\mathbf{v}$ is an ordered } &\text{basis of $H_2(Z;\mathbb{Z})$}\\
 &\text{whose intersection matrix is $Q$.}\}. 
\end{align*}
\end{definition}
We call $G_{Z,Q}$ the \textit{intersection genus} of $Z$ with respect to $Q$ (\textit{$Q$-genus} of $Z$ for short). It is straightforward to see that the $Q$-genus is a diffeomorphism invariant of 4-manifolds for any fixed intersection matrix $Q$. Namely, if 4-manifolds $Z$ and $Z'$ are orientation-preserving diffeomorphic to each other, then $G_{Z,Q}=G_{Z',Q}$ for any intersection matrix $Q$.

\subsection{Coarsely distinguishing smooth structures}\label{subsec:distinguish}
For an element $\alpha$ in a finitely generated abelian group $G$, we define the \textit{divisibility} $d(\alpha)$ of $\alpha$ by 
\begin{equation*}
d(\alpha)=\left\{
\begin{array}{ll}
\max\{n\in \mathbb{Z}\mid \text{$[\alpha]=n\alpha'$ for some $\alpha'\in G/\textnormal{Tor}$}\},&\text{if $\alpha$ is not torsion;}\\
0, &\text{if $\alpha$ is torsion.}
\end{array}
\right.
\end{equation*}
Note that $d(\alpha)$ is a non-negative integer. 
\begin{remark}We use this unnatural definition in order to relax the assumption of Theorem~\ref{thm:Stein:exotic} and to obtain stronger estimates of intersection genera in the proof. 
This theorem also works under the natural definition below. 
\begin{equation*}
d(\alpha)=\left\{
\begin{array}{ll}
\max\{n\in \mathbb{Z}\mid \text{$\alpha=n\alpha'$ for some $\alpha'\in G$}\},&\text{if $\alpha$ is not torsion;}\\
0, &\text{if $\alpha$ is torsion.}
\end{array}
\right.
\end{equation*}
\end{remark}

Here we prove Theorem~\ref{thm:Stein:exotic} using intersection genera. 

%%%%begin: thm restate%%%
\newtheorem*{thm:Stein:exotic:restate}{Theorem~\ref{thm:Stein:exotic}}
%%%%begin: thm restate%%%%

\begin{thm:Stein:exotic:restate}For a given infinite family of compact Stein 4-manifolds, if the divisibilities of the first Chern classes of the members are pairwise distinct, then at least infinitely many members are pairwise non-diffeomorphic. 
\end{thm:Stein:exotic:restate}
\begin{proof}Let $\{X_i\mid i\in \mathbb{N}\}$ be an infinite family of compact Stein 4-manifolds in the assumption of this theorem, and let $d_i$ be the divisibility of $c_1(X_i)$. Without loss of generality, we may assume that $\{d_i\}_{i\in \mathbb{N}}$ is a strictly increasing sequence. It suffices to prove the theorem in the case where all the members have the same intersection form. We put $n=b_2(X_1)$ and fix an intersection matrix $Q$ of $X_1$. The assumption on the divisibilities guarantees $n\geq 1$. Note that each $H_2(X_i;\mathbb{Z})$ has no torsion, since any Stein 4-manifold has a handle decomposition without 3- and 4-handles due to \cite{E1} (cf.\cite{G1}). Let $m_1,m_2,\dots,m_n$ be the diagonal components of $Q$. 

We here estimate the value, denoted by $G_i$, of the $Q$-genus of each $X_i$. Let $\mathbf{v}=\{v_1,v_2,\dots,v_n\}$ be an ordered basis of $H_2(X_i;\mathbb{Z})$ whose intersection matrix is the aforementioned matrix $Q$. Since $d_i$ is the divisibility of $c_1(X_i)$, one can check that the inequality 
$\left| \langle c_1(X_i), v_{j} \rangle \right|\geq d_{i}$ 
holds for at least one $j$ with $1\leq j \leq n$.  Applying the adjunction inequality to $X_i$ and this $v_j$, we see that
\begin{equation*}
d_i+m_j\leq 2g_{X_i}(v_j)-2.
\end{equation*}
Hence, for each $i$, we obtain the estimate
\begin{equation*}
d_i+\min\{m_j\mid 1\leq j \leq n\}\leq 2G_i-2.
\end{equation*}

Since $\lim _{i\to \infty}d_i=\infty$, this estimate implies $\lim _{i\to \infty}G_i=\infty$. %Therefore, there exists a strictly increasing integer sequence $\{i_k\}_{k\in \mathbb{N}}$ such that $\{G_{i_k}\}_{k\in \mathbb{N}}$ is also a strictly increasing sequence. 
Therefore, there exists a strictly increasing subsequence $\{G_{i_k}\}_{k\in \mathbb{N}}$ of $\{G_{i}\}_{i\in \mathbb{N}}$. 
Since the $Q$-genus $G_i$ is a diffeomorphism invariant of $X_i$, the members of the infinite subfamily $\{X_{i_k}\mid k\in \mathbb{N}\}$ are pairwise non-diffeomorphic. 
\end{proof}

\begin{remark}This proof gives a simple method for obtaining lower bounds of intersection genera of given compact Stein 4-manifolds, since one can calculate their first Chern classes from their Stein handlebody structures (\cite{G1}) and Lefschetz fibration structures (\cite{EtOz2}). Also, one can give upper bounds of intersection genera by using their smooth handlebody structures. Thus we can often use intersection genera to distinguish the smooth structures of two given Stein 4-manifolds. Indeed, this argument effectively works for many examples of exotic Stein 4-manifolds obtained in \cite{AY6} and \cite{Y6}. 
\end{remark}

We can now easily prove Theorem~\ref{intro:thm:main}. 
%%%%%%%%%%%%
\newtheorem*{intro:thm:main:restate}{Theorem~\ref{intro:thm:main}}
%%%%%%%%%%%%

\begin{intro:thm:main:restate}
Let $\{X_i\mid i\in \mathbb{N}\}$ be an infinite family of compact Stein 4-manifolds such that, for each $i$, the contact 5-manifold $(W_i,\xi_i)$ supported by the contact open book $(X_i, \textnormal{id})$ admits a subcritical Stein filling without 1-handles. If the members of $\{(W_i,\xi_i)\mid i\in \mathbb{N}\}$ are pairwise non-contactomorphic, then at least infinitely many members of $\{X_i\mid i\in \mathbb{N}\}$ are pairwise non-diffeomorphic. 
\end{intro:thm:main:restate}

\begin{proof}
%We first observe that $H^2(W_i;\mathbb{Z})$ is isomorphic to $H^2(X_i;\mathbb{Z})$ for each $i$. Since $W_i$ is the boundary of $X_i\times D^2$, the cohomology exact sequence gives the exact sequence \begin{equation}H^2(X_i\times D^2,W_i;\mathbb{D^2})\to H^2(X_i\times D^2;\mathbb{D^2})\to H^2(W_i;\mathbb{D^2})\to H^2(X_i\times D^2,W_i;\mathbb{D^2})\end{equation}
%We consider a Since $X_i$ is a Stein manifold, $X_i$ has a handle decomposition without 3- and 4-handles due to \cite{E1} (cf.\cite{G1}). We thus see $H^3(X_i\times D^2;\mathbb{Z})=0$. $H^3(X_i\times D^2;\mathbb{Z})=0$
%Since $W_i=\partial (X_i\times D^2)=\partial (X_i\times D^1\times D^1)$, we easily see that $W_i$ is diffeomorphic to the double $(X_i\times D^1)\cup \overline{X_i\times D^1}$ of $X_i\times D^1$, where $\overline{X_i\times D^1}$ denotes $X_i\times D^1$ with the reverse orientation. Since $X_i$ is a Stein manifold, $X_i$ has a handle decomposition without 3- and 4-handles due to \cite{E1} (cf.\cite{G1}). This implies $b_2(W_i)=b_2(X_i\times D^1)=b_2(X_i)$. 

It suffices to prove the case where all of $H^2(X_i; \mathbb{Z})$'s are pairwise isomorphic. We here observe that each $X_i$ is simply connected. %each $H^2(X_i; \mathbb{Z})$ has no torsion. %Since $(W_i,\xi_i)$ admits a subcritical Stein filling without 1-handle, $(W_i,\xi_i)$ is supported by the contact open book $(Y_i,id)$ for some Stein handlebody $Y_i$ without 1-handles. 
Clearly $W_i$ is the boundary of $X_i\times D^2=(X_i\times D^1)\times D^1$. It thus follows that each $W_i$ is diffeomorphic to the double $(X_i\times D^1)\cup \overline{X_i\times D^1}$ of $X_i\times D^1$. %, since %Here $\overline{X_i\times D^1}$ denotes $X_i\times D^1$ with the opposite orientation. 
Since any Stein 4-manifold admits a handle decomposition without 3- and 4-handles, one can see that $\pi_1(W_i)$ is isomorphic to $\pi_1(X_1\times D^1)\cong \pi_1(X_i)$. Since $(W_i,\xi_i)$ admits a simply connected subcritical Stein filling, Remark~\ref{rem:subcritical} shows that $(W_i,\xi_i)$ is supported by the contact open book $(Y_i,\textnormal{id})$ for some simply connected Stein handlebody $Y_i$. Note that $\pi_1(Y_i)$ is isomorphic to $\pi_1(W_i)\cong \pi_1(X_i)$. Therefore, $X_i$ is simply connected.

By Theorem~\ref{thm:DGV} and Lemma~\ref{lem:iso}, we see that, for any $i\neq j$, there exists no isomorphism between $H^2(X_i;\mathbb{Z})$ and $H^2(X_j;\mathbb{Z})$ that maps $c_1(X_i)$ to $c_1(X_j)$. Since $X_i$ is simply connected, the universal coefficient theorem tells that $H^2(X_i; \mathbb{Z})$ has no torsion. We thus easily see that the divisibilities of $c_1(X_i)$ and $c_1(X_j)$ are not equal to each other. Therefore Theorem~\ref{thm:Stein:exotic} tells that at least infinitely many of $X_i$'s are pairwise non-diffeomorphic. 
\end{proof}
\begin{remark}As seen from this proof, for a compact Stein 4-manifold $X$, if the contact 5-manifold $(W,\xi)$ supported by the contact open book $(X, \textnormal{id})$ admits a simply connected subcritical Stein filling, then $X$ is simply connected. Van Koert kindly pointed out the following generalization: for a contact 5-manifold $(W,\xi)$ supported by a contact open book $(X, \textnormal{id})$, any Stein filling $V$ of $(W,\xi)$ satisfies $\pi_1(V)\cong \pi_1(X)$. To see this, we note that $\pi_1(V)\cong \pi_1(\partial V)$ holds for any compact Stein manifold $V$ of real dimension at least 6 (Lemma~3.1 in \cite{OV}). Applying this lemma twice, we see 
\begin{equation*}
\pi_1(V)\cong \pi_1(W)\cong \pi_1(X\times D^2)\cong \pi_1(X), 
\end{equation*}
and thus the above generalization follows. 
\end{remark}

The obvious corollary below says that a given compact smooth 4-manifold with boundary induces at most finitely many contact 5-manifolds admitting subcritical Stein fillings without 1-handles by contact open books with the identity monodromy. 

\begin{corollary}For any simply connected compact oriented smooth 4-manifold $X$ with boundary, there are at most finitely many contact 5-manifolds, up to contactomorphism, admitting subcritical Stein fillings without 1-handles such that each of them is supported by a contact open book $(Y, \textnormal{id})$ for a compact Stein 4-manifold $Y$ diffeomorphic to $X$. 
\end{corollary}

It would be natural to ask whether the condition `admitting subcritical Stein fillings without 1-handles' can be removed from this corollary. %If this question has a negative answer, then Question~\ref{question:Andrews} also has a negative answer due to Lemma~\ref{lem:AC}.%
If Question~\ref{question:Andrews} has an affirmative answer, then this question is also affirmative due to Lemma~\ref{lem:AC}. %If this extension does not hold for a simply connected $X$, then there exist counterexamples to the following question. For any finite presentation of the trivial group, can all the generators of the presentation are eliminated by the Andrews-Curtis moves? This question is a generalization of the Andrews-Curtis conjecture which is a well-known open problem. 

\subsection{Finiteness of the number of the first Chern classes}
Here we prove the following proposition. %Indeed, this result gives an alternative proof of Theorem~\ref{thm:Stein:exotic}. 

\begin{proposition}\label{prop:finiteness:exotic}For any compact connected oriented smooth 4-manifold $X$ with boundary, there are at most finitely many classes of $H^2(X;\mathbb{Z})$ that are the first Chern classes of Stein structures on $X$.
\end{proposition}
\begin{proof}%Let $X$ be a compact connected oriented smooth 4-manifold with boundary. 
If $b_2(X)=0$, then $H^2(X;\mathbb{Z})$ is a finite set, implying the proposition. We thus assume $b_2(X)>0$. Assume further that $X$ admits a Stein structure. By \cite{E1} (cf.\ \cite{G1}, \cite{GS}), $X$ admits a Stein handlebody decomposition. Note that this handlebody has neither 3- nor 4-handles. %Let $h_1^{(1)}, h_2^{(1)}, \dots, h_m^{(1)}$ be the 1-handles of the handlebody, and let $h_1^{(2)}, h_2^{(2)}, \dots, h_m^{(2)}$ be the 2-handles of the handlebody. 

We consider the CW complex given by the handlebody and discuss its cohomology and homology groups. %We regard each $i$-handle as an $i$-chain. 
Let $C_i(X)$ denote the $i$-th chain group, and let $\partial_i: C_i(X)\to C_{i-1}(X)$ be the boundary homomorphism. By changing bases, we have a basis $v_1,v_2,\dots, v_n$ of $C_2(X)$, a basis $u_1,u_2,\dots, u_l$ of $C_1(X)$, and non-zero integers $a_{m+1}, \dots, a_n$ such that 
\begin{equation*}
\partial_2(v_i)=\left\{
\begin{array}{ll}
0, &\text{if $1\leq i\leq m$;}\\
a_iu_i, &\text{if $m+1\leq i\leq n$.}
\end{array}
\right.
\end{equation*}
We thus see that $[v_1], [v_2],\dots, [v_m]$ is a basis of $H_2(X;\mathbb{Z})$. We note that the assumption $b_2(X)\geq 1$ guarantees $m\geq 1$. Here let $C^i(X)$ denote the $i$-th cochain group, and let $v_i^*: C_2(X)\to \mathbb{Z}$ be the dual of $v_i$ for each $i$. Clearly $\{[v_1^*], [v_2^*], \dots, [v_n^*]\}$ is a generating set of $H^2(X;\mathbb{Z})$. Note that, for $m+1\leq i\leq n$, the class $[v_{i}^*]$ is a torsion element of order $|a_{i}|$ . 

Now let $c$ be the first Chern class of a Stein structure on $X$. Then there exist integers $c_1, c_2,\dots, c_n$ such that $c=c_1[v_1^*]+c_2[v_2^*]+\dots+c_n[v_n^*]$. We may assume that $|c_i|\leq a_i$ for $m+1\leq i\leq n$, since $[v_{i}^*]$ is of order $|a_i|$. Let $g_i$ be the genus of a fixed smoothly embedded surface in $X$ representing the second homology class $[v_i]$. By applying the adjunction inequality to the first Chern class $c$ and the second homology class $[v_i]$, we obtain $|c_i|+[v_i]\cdot [v_i]\leq 2g_i-2$ for $1\leq i\leq m$. Therefore, only finitely many integers can be the value of $c_i$ for each $1\leq i\leq n$. This shows that there are at most finitely many classes of $H^2(X;\mathbb{Z})$ that can be the first Chern classes of Stein structures on $X$. 
\end{proof}

This proposition gives an alternative proof of Theorem~\ref{thm:Stein:exotic} by the pigeonhole principle. However, the previous proof is much more practical, since this alternative proof does not tell how to distinguish smooth structures of given two Stein 4-manifolds in the infinite family. 
\begin{remark}For a compact oriented smooth 4-manifold $X$ with boundary, let $N_C(X)$ denote the number of the first Chern classes of Stein structures on $X$. This is clearly a diffeomorphism invariant of $X$. The proof of Proposition~\ref{prop:finiteness:exotic} tells how to give an upper bound of $N_C(X)$, and  one can give a lower bound of $N_C(X)$  by finding Stein structures on $X$ (e.g.\ Stein handlebodies and Lefschetz fibrations). For example, using the algorithm obtained in \cite{Y6}, we can construct infinite families of pairwise exotic 4-manifolds which can be distinguished by $N_C$ (e.g.\ exotic 4-manifolds used in Theorem~\ref{thm:exotic} of this paper). However, compared with the intersection genus, we need more efforts for obtaining upper and lower bounds of $N_C$, and thus the intersection genus is more practical and useful for coarsely distinguishing smooth structures. 
\end{remark}
%%%%%%%%%%%%%%%%%%%%%%%%%%
%%%%%%%%%%%%%%%%%%%%%%%%%%%
\section{Examples}
In this section, we construct two examples motivated by Theorem~\ref{intro:thm:main}. The first example below tells that `at least infinitely many members' in the conclusion of Theorem~\ref{intro:thm:main} cannot be strengthened to `all members' without additional constraints. 

\begin{proposition}\label{prop:strengthen}Let $n\geq 2$ be a positive integer, and let $(W_1,\xi_1)$, $(W_2,\xi_2)$, \dots, $(W_n,\xi_n)$ %$\{(W_i,\xi_i)\mid i=1,2,\dots, n\}$  be a finite family of 
be pairwise non-contactomorphic closed contact 5-manifolds admitting subcritical Stein fillings without 1-handles. If they are pairwise diffeomorphic smooth 5-manifolds, then there exist 
%a finite family $\{X_i\mid i=1,2,\dots, n\}$ of 
pairwise diffeomorphic compact Stein 4-manifolds $Z_1, Z_2, \dots, Z_n$ such that each contact open book $(Z_i, \textnormal{id})$ supports $(W_i,\xi_i)$. 
\end{proposition}
\begin{proof}By Remark~\ref{rem:classification}, we see that there exist a positive integer $k$ and non-negative integers $r_1,r_2,\dots,r_n$ such that each $(W_i,\xi_i)$ is contactomorphic to $(S_{k,r_i}, \zeta_{k,r_i})$. Note that all of the $r_i$'s have the same parity, since $W_i$'s are pairwise diffeomorphic. Let us fix a smooth knot $K$ in $S^3$ satisfying the following condition: %Here we fix a sufficiently large integer $p$, and let $K$ be the $(p,2)$-torus knot in $S^3$. 
for each $1\leq i\leq n$, the knot $K$ has a Legendrian representative $\mathcal{L}_i$ whose rotation number is $r_i$ and whose Thurston-Bennequin number is a fixed number $r_1+1$ (For example, the $(p,2)$-torus knot with a sufficiently large $p$ is such an example of $K$.). Using this Legendrian representative, we define a Stein 4-manifold $X_i$ as the Stein handlebody obtained from $D^4$ by attaching a 2-handle along $\mathcal{L}_i$. Note that the framing $r_1$ is independent of $i$. We also use a Stein handlebody $Y_{k-1}$ without 1-handles satisfying $b_2(Y_{k-1})=k-1$ and $r(Y_{k-1})=0$. One can easily find such a $Y_{k-1}$. Now let $Z_i$ be the boundary connected sum $X_i\natural Y_{k-1}$, which is a Stein handlebody. Note that $Z_i$'s are pairwise diffeomorphic. 
Since $r(Z_i)=r_i$ and $b_2(Z_i)=k$, the contact 5-manifold $(S_{k,r_i}, \zeta_{k,r_i})$ is supported by the contact open book $(Z_i, \textnormal{id})$. Therefore the claim follows. 
\end{proof}

The next example below tells that there are many examples of infinitely many pairwise exotic Stein 4-manifolds whose smooth structures are distinguished by Theorem~\ref{intro:thm:main}. %This example easily follows from the proof of Theorem~1.2 in \cite{AY7}, which used exotic 4-manifolds obtained in \cite{AY6} and \cite{Y6}. 
Let us recall Remark~\ref{rem:classification}. Any closed contact 5-manifold with $b_2=n$ admitting a subcritical Stein filling without 1-handles is diffeomorphic to either $\#_nS^2\times S^3$ or $\#_nS^2\widetilde{\times} S^3$.  Furthermore, the set of contactomorphism classes of such contact structures on $\#_nS^2\times S^3$ (resp.\ $\#_nS^2\widetilde{\times} S^3$) is given by $\{\zeta_{2r}\mid r\in \mathbb{Z}_{\geq 0}\}$ (resp.\ $\{\zeta_{2r+1}\mid r\in \mathbb{Z}_{\geq 0}\}$), where $\mathbb{Z}_{\geq 0}$ denotes the set of non-negative integers. %We show that these infinitely many distinct contact structures have pairwise homeomorphic Stein 4-manifolds as pages of their supporting contact open books. 

\begin{theorem}[cf.\ \cite{AY7}]\label{thm:exotic}For each fixed integer $n\geq 2$, there exists an infinite family $\{Z_{n,2r}\mid r\in \mathbb{Z}_{\geq 0}\}$ of pairwise homeomorphic compact Stein 4-manifolds such that each contact open book $(Z_{n,2r}, \textnormal{id})$ supports $(\#_nS^2\times S^3, \zeta_{2r})$. $($Hence, by Theorem~\ref{intro:thm:main}, at least infinitely members are pairwise homeomorphic but non-diffeomorphic.$)$ Furthermore, the same claim also holds for $\{(\#_nS^2\widetilde{\times} S^3, \zeta_{2r+1})\mid r\in \mathbb{Z}_{\geq 0}\}$. 
\end{theorem}

\begin{proof}The following 4-manifolds and Stein handlebodies were obtained in \cite{AY6}, \cite{Y6} and \cite{AY7}, and this proof is based on the proof of Theorem~1.3 in \cite{AY7}. %For details of the construction and their properties, see \cite{AY7}. 

For a non-negative integer $p$, let $X_p$ be the smooth 4-manifold given by the handlebody diagram in Figure~\ref{fig:smooth_nuclei}, where the box $p$ denotes $p$ right-handed full twists. We note that each $X_p$ is obtained from $X_0$ by a log transform with multiplicity one. By canceling the 1-handles, we have the diagrams of $X_p$ in Figure~\ref{fig:2-handle_nuclei}. By isotopy, we obtain the Stein handlebody decomposition of $X_p$ in Figure~\ref{fig:many_contact}. One can check that the rotation numbers of the attaching circles (Legendrian knots) of two 2-handles of $X_p$ are $0$ and $p$ for each $p\geq 0$. 

Now let $Y_{n-2}$ be a 4-dimensional Stein handlebody without 1-handles satisfying $b_2(Y_{n-2})=n-2$ and $r(Y_{n-2})=0$, and let $Z_{n,p}$ be the boundary connected sum $X_p\natural Y_{n-2}$, which is a Stein handlebody. 
 Since $r(Z_{n,p})=p$ and $b_2(Z_{n,p})=n$, we see that the contact open book $(Z_{n,p}, \textnormal{id})$ supports the contact 5-manifold $(S_{n,p}, \zeta_{n,p})$. 

Lastly we show that, for any non-negative integers $r$ and $r'$, 4-manifolds $Z_{n,2r}$ and $Z_{n,2r'}$ are homeomorphic to each other. We can easily check that the intersection forms of $X_{n,2r}$ and $X_{n,2r'}$ are even, unimodular and indefinite. By the classification of indefinite unimodular intersection forms (cf.\ \cite{GS}), we see that the intersection forms of $X_{n,2r}$ and $X_{n,2r'}$ are isomorphic. Since they are simply connected, and their boundaries are homeomorphic to each other, Freedman's theorem~\cite{Fr} (cf.\ \cite{B}) tells that any diffeomorphism  $\partial X_{n,2r}\to \partial X_{n,2r'}$ extends to a homeomorphism $X_{n,2r}\to X_{n,2r'}$. Since $Z_{n,2r'}$ is obtained from $Z_{n,2r}$ by removing $X_{n,2r}$ and gluing $X_{n,2r'}$, it follows that $Z_{n,2r}$ and $Z_{n,2r'}$ are homeomorphic to each other. Therefore the theorem holds for the family $\{(\#_nS^2\times S^3, \zeta_{2r})\mid r\in \mathbb{Z}_{\geq 0}\}$. By the same argument, we see that the theorem also holds for $\{(\#_nS^2\widetilde{\times} S^3, \zeta_{2r+1})\mid r\in \mathbb{Z}_{\geq 0}\}$. 
\end{proof}

\begin{figure}[h!]
\begin{center}
\includegraphics[width=1.3in]{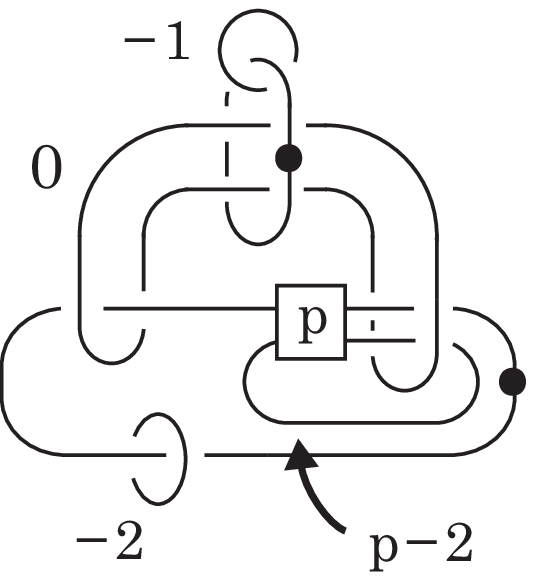}
\caption{$X_p$}
\label{fig:smooth_nuclei}
\end{center}
\end{figure}

\begin{figure}[h!]
\begin{center}
\includegraphics[width=3.1in]{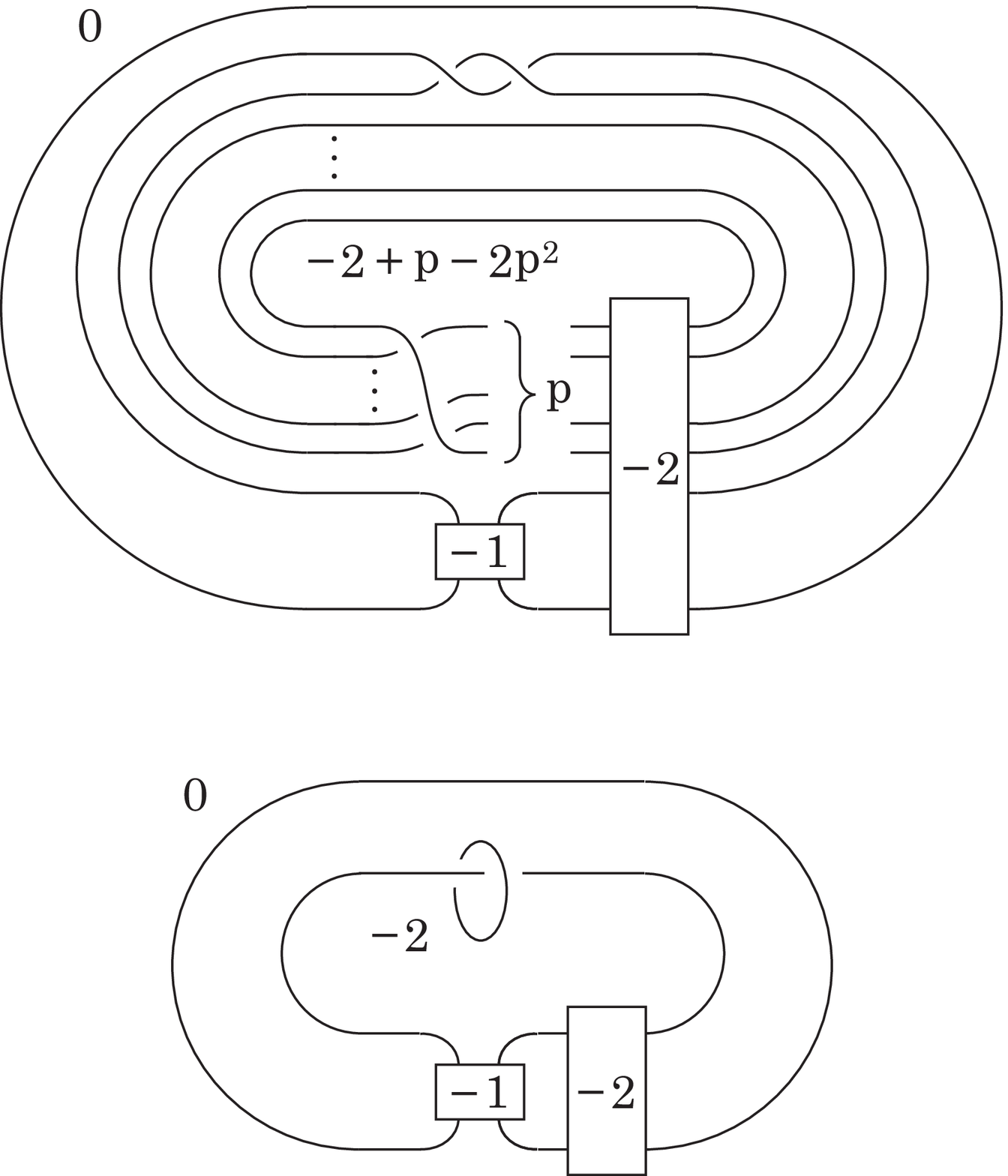}
\caption{The upper handlebody is $X_p$ for $p\geq 1$, and the lower one is $X_0$.}
\label{fig:2-handle_nuclei}
\end{center}
\end{figure}

\begin{figure}[h!]
\begin{center}
\includegraphics[width=2.7in]{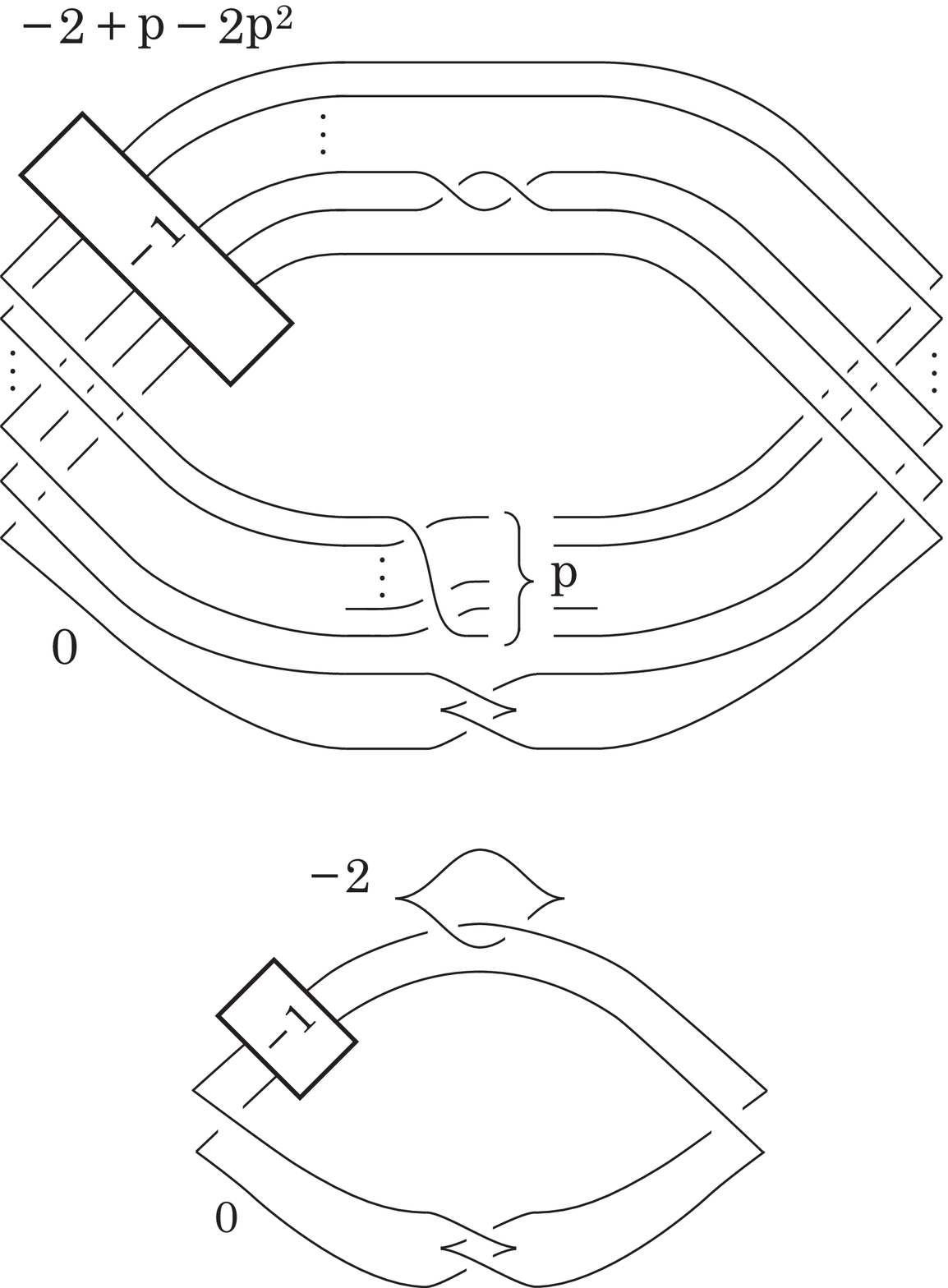}
\caption{The upper Stein handlebody is $X_p$ for $p\geq 1$, and the lower one is $X_0$.}
\label{fig:many_contact}
\end{center}
\end{figure}

\begin{figure}[h!]
\begin{center}
\includegraphics[width=2.1in]{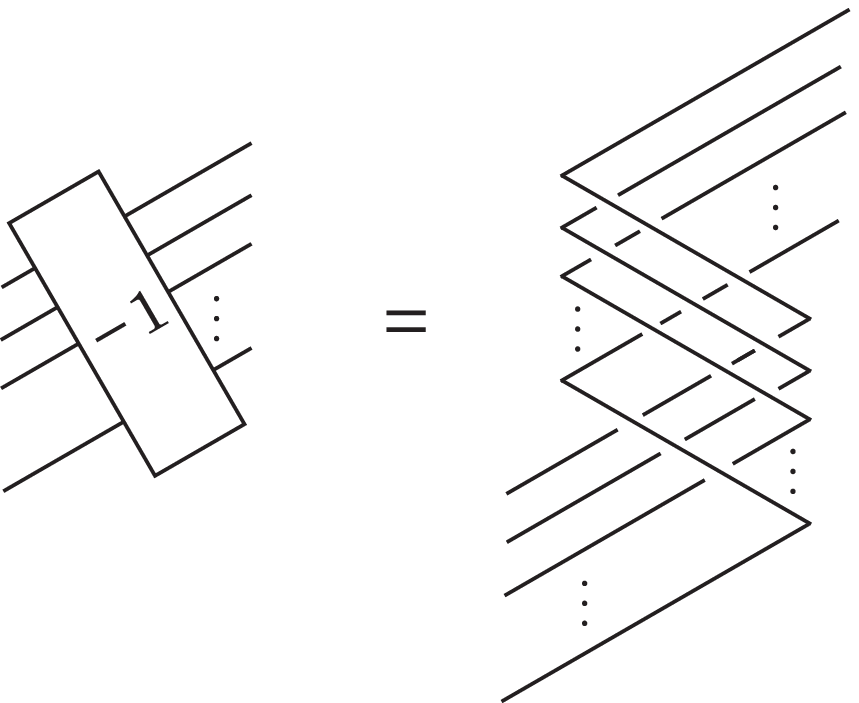}
\caption{Legendrian version of left-handed full twists}
\label{fig:Legendrian_left_twists}
\end{center}
\end{figure}

\subsection*{Acknowledgements}The author would like to thank Otto van Koert for his helpful comments. 
%%%%%%%%%%%%%%%%%%%%%%%%%%%%%%%%%%%%%%%%%%%%%%%%%%%%%%%%%
%%%%%%%%%%%%%%%%%%%%%%%%%%%%%%%%%%%%%%%%%%%%%%%%%%%%%inyouowari


\begin{thebibliography}{35}
%\bibitem{A1}S.\ Akbulut, \textit{A fake compact contractible $4$-manifold}, J.\ Differential Geom.\ \textbf{33} (1991), no.\ 2, 335--356.
%\bibitem{A6}S.\ Akbulut, \textit{A solution to a conjecture of Zeeman}, Topology, vol.30, no.3, (1991), 513--515. 
%\bibitem{A2}S.\ Akbulut, \textit{An exotic $4$-manifold}, J.\ Differential Geom.\ \textbf{33} (1991), no.\ 2, 357--361.
%\bibitem{A3}S.\ Akbulut, \textit{Constructing a fake $4$-manifold by Gluck construction  to a standard $4$-manifold,}, Topology, vol.27, no.\ 2 (1988), 239-243.


%\bibitem{AM}S.\ Akbulut and R.\ Matveyev, \textit{Exotic structures and adjunction inequality}, Turkish J.\ Math.\ \textbf{21} (1997), no.\ 1, 47--53.
%\bibitem{AM2}S.\ Akbulut and R.\ Matveyev, \textit{A convex decomposition theorem for $4$-manifolds}, Internat. Math. Res. Notices \textbf{1998}, no.\ 7, 371--381.
%\bibitem{AO1}S.\ Akbulut and B.\ Ozbagci, \textit{Lefschetz fibrations on compact Stein surfaces}, Geom.\ Topol.\ 5 (2001), 319--334.
%\bibitem{AO2}S.\ Akbulut and B.\ Ozbagci, \textit{Erratum: ``Lefschetz fibrations on compact Stein surfaces'' \textnormal{[Geom.\ Topol.\ 5 (2001), 319--334.]}}, Geom.\ Topol.\ 5 (2001), 939--945
%\bibitem{AO3}S.\ Akbulut and B.\ Ozbagci, \textit{On the topology of compact Stein surfaces}, Int.\ Math.\ Res.\ Not.\ 2002, no.\ 15, 769--782.

\bibitem{AM}S.\ Akbulut and R.\ Matveyev, \textit{Exotic structures and adjunction inequality}, Turkish J.\ Math.\ \textbf{21} (1997) 47--53.

%\bibitem{AY1}S.\ Akbulut and K.\ Yasui, \textit{Corks, Plugs and exotic structures}, J.\ G\"{o}kova Geom.\ Topol.\ GGT \textbf{2} (2008), 40--82.  
%\bibitem{AY3}S.\ Akbulut and K.\ Yasui, \textit{Knotting corks},  J.\ Topol.\ \textbf{2} (2009), no.\ 4, 823--839. 
%\bibitem{AY2}S.\ Akbulut and K.\ Yasui, \textit{Small exotic Stein manifolds}, Comment.\ Math.\ Helv.\ \textbf{85} (2010), no.\ 3, 705--721. 
%\bibitem{AY4}S.\ Akbulut and K.\ Yasui, \textit{Stein 4-manifolds and corks}, J.\ G\"{o}kova Geom.\ Topol.\ GGT \textbf{6} (2012), 58--79. 
\bibitem{AY5}S.\ Akbulut and K.\ Yasui, \textit{Cork twisting exotic Stein 4-manifolds}, J.\ Differential Geom.\ \textbf{93} (2013), no. 1, 1--36. 
\bibitem{AY6}S.\ Akbulut and K.\ Yasui, \textit{Infinitely many small exotic Stein fillings}, arXiv:1208.1053, to appear in Journal of Symplectic Geometry. 
\bibitem{AY7}S.\ Akbulut and K.\ Yasui, \textit{Contact 5-manifolds admitting open books with exotic pages},  arXiv:1502.06118, to appear in Math.\ Res.\ Lett.
\bibitem{AC}J.\,J.\ Andrews and M.\,L.\ Curtis, \textit{Free groups and handlebodies}, Proc.\ Amer.\ Math.\ Soc.\ \textbf{16} (1965), 192--195. 
\bibitem{C} K.\ Cieliebak, \textit{Subcritical Stein manifolds are split}, arXiv:math/0204351v1.
%\bibitem{AY7}S.\ Akbulut and K.\ Yasui, \textit{Cork twisting exotic Stein 4-manifolds II}, in preparation. 

%\bibitem{AEMS}A.\ Akhmedov, J.\ B.\ Etnyre, T.\ E.\ Mark, and I. Smith, \textit{A note on Stein fillings of contact manifolds}, Math. Res. Lett. \textbf{15} (2008), no. 6, 1127--1132. 
%\bibitem{AkhOz1}A.\ Akhmedov and B.\ Ozbagci, \textit{Singularity links with exotic Stein fillings},  J.\ Singul.\ \textbf{8} (2014), 39--49. 
%\bibitem{AkhOz2}A.\ Akhmedov and B.\ Ozbagci, \textit{Exotic Stein fillings with arbitrary fundamental group}, arXiv:1212.1743v1. 
%\bibitem{Ba}D.\ Barden, \textit{Simply connected five-manifolds}, Ann.\ of Math.\ \textbf{82} (1965), 365--385.
\bibitem{B}S.\ Boyer, \textit{Simply-connected $4$-manifolds with a given boundary}, Trans. Amer. Math. Soc. \textbf{298} (1986), no.\ 1, 331--357.
\bibitem{CE}K.\ Cieliebak and Y.\ Eliashberg, \textit{From Stein to Weinstein and back. Symplectic geometry of affine complex manifolds}, American Mathematical Society Colloquium Publications \textbf{59}, American Mathematical Society, Providence, RI, 2012. xii+364 pp.
%\bibitem{CGH}V.\ Colin, E.\ Giroux and K.\ Honda \textit{Finitude homotopique et isotopique des structures de contact tendues} (French) [Homotopy and isotopy finiteness of tight contact structures] Publ.\ Math.\ Inst.\ Hautes Etudes Sci.\ No.\ \textbf{109} (2009), 245--293. 
%\bibitem{C}C.\ L.\ Curtis, M.\ H. Freedman, W.\ C.\ Hsiang, and R.\ Stong, \textit{A decomposition theorem for $h$-cobordant smooth simply-connected compact $4$-manifolds}, Invent. Math. \textbf{123} (1996), no.\ 2, 343--348.

%\bibitem{DG1} F.\ Ding and H.\ Geiges, \textit{Symplectic fillability of tight contact structures on torus bundles}, Algebr.\ Geom.\ Topol.\ \textbf{1} (2001), 153--172.
%\bibitem{DG2} F.\ Ding and H.\ Geiges, \textit{A Legendrian surgery presentation of contact 3-manifolds}, Proc.\ Cambridge Philos.\ Soc.\ \textbf{136} (2004) 583--598.
%\bibitem{DG}F.\ Ding and H.\ Geiges, \textit{The diffeotopy group of $S^1\times S^2$ via contact topology}, Compos. Math. \textbf{146} (2010), no.\ 4, 1096--1112.
\bibitem{DGV}F.\ Ding, H.\ Geiges and O.\ van Koert, \textit{Diagrams for contact 5-manifolds}, J.\ Lond.\ Math.\ Soc.\ \textbf{86} (2012),  no.\ 3, 657--682.
\bibitem{E1}Y.\ Eliashberg. \textit {Topological characterization of Stein manifolds of dimension $>2$}, International J. of Math. Vol. 1 (1990), No 1  pp. 29-46.

%\bibitem{E2}Y.\ Eliashberg. \textit{Filling by holomorphic discs and its applications}, Geometry of low-dimensional manifolds, 2 (Durham, 1989), 45--67, London Math. Soc. Lecture Note Ser., 151, Cambridge Univ. Press, Cambridge, 1990. 
%\bibitem{ElP}Y.\ Eliashberg and L.\ Polterovich, \textit{New applications of Luttinger's surgery}, Comment.\ Math.\ Helv.\ \textbf{69} (1994), no. 4, 512-522.
%\bibitem{E3}Y.\ Eliashberg. \textit{A few remarks about symplectic filling}, Geom.\ Topol.\ \textbf{8} (2004) 277--293. 
%\bibitem{Et1}J.\ Etnyre, \textit{Introductory lectures on contact geometry}, Topology and geometry of manifolds (Athens, GA, 2001),  81--107, Proc.\ Sympos.\ Pure Math., \textbf{71}, Amer.\ Math.\ Soc., Providence, RI, 2003. 
%\bibitem{Et2}J.\ Etnyre, \textit{On symplectic fillings}, Algebr.\ Geom.\ Topol.\ \textbf{4} (2004),  73--80.
%\bibitem{Et4}J.\ Etnyre, \textit{Lectures on open book decompositions and contact structures}, Floer homology, gauge theory, and low-dimensional topology, 103--141, Clay Math.\ Proc., \textbf{5}, Amer.\ Math.\ Soc., Providence, RI, 2006.
%\bibitem{Et5}J.\ Etnyre, \textit{Contact geometry in low dimensional topology}, Low dimensional topology, 229--264, IAS/Park City Math.\ Ser., \textbf{15}, Amer.\ Math.\ Soc., Providence, RI, 2009. 
%\bibitem{EH}J.\ Etnyre and K.\ Honda, \textit{On symplectic cobordisms}, Math.\ Ann.\ \textbf{323} (2002) 31--39.
\bibitem{EtOz2}J.\ Etnyre and B.\ Ozbagci, \textit{Invariants of contact structures from open books}, Trans.\ Amer.\ Math.\ Soc.\ \textbf{360} (2008),  no.\ 6, 3133--3151.
\bibitem{FS3}R.\ Fintushel and R.\ J.\ Stern, \textit{Immersed spheres in 4-manifolds and the immersed Thom conjecture}, Turkish J.\ Math.\ \textbf{19} (1995), 145--157.
\bibitem{Fr}M.\ Freedman, \textit{The topology of four-dimensional manifolds}, J.\ Differential Geom.\ \textbf{17} (1982), no.\ 3, 357--453.
%\bibitem{Ge}H.\ Geiges, \textit{An introduction to contact topology}, Cambridge Studies in Advanced Mathematics, \textbf{109}. Cambridge University Press, Cambridge, 2008. xvi+440 pp.
\bibitem{Gi} E.\ Giroux, \textit{G\'{e}om\'{e}trie de contact: de la dimension trois vers les dimensions sup\'{e}rieures} (French) [Contact geometry: from dimension three to higher dimensions], Proceedings of the International Congress of Mathematicians, Vol.\ II (Beijing, 2002),  405--414, Higher Ed.\ Press, Beijing, 2002.
\bibitem{G1}R.\,E.\ Gompf, \textit{Handlebody construction of Stein surfaces}, Ann. of Math.\ (2) \textbf{148} (1998), no. 2, 619--693.
%\bibitem{G_AGT}R.\ E.\ Gompf, \textit{More Cappell-Shaneson spheres are standard}, Algebr.\ Geom.\ Topol.\ \textbf{10} (2010),  no.\ 3, 1665--1681.
\bibitem{G_GT}R.\,E.\ Gompf, \textit{Minimal genera of open 4-manifolds}, arXiv:1309.0466, to appear in Geom.\ Topol.
\bibitem{GS}R.\,E.\ Gompf and A.\,I.\ Stipsicz, \textit{$4$-manifolds and Kirby calculus}, Graduate Studies in Mathematics, \textbf{20}. American Mathematical Society, 1999.
%\bibitem{LM}P.\ Lisca and G.\ Mati\'c, \textit{Tight contact structures and Seiberg-Witten invariants}, Invent. Math. \textbf{129} (1997), no. 3, 509--525.
\bibitem{KM1}P.\ Kronheimer and T.\ Mrowka, \textit{The genus of embedded surfaces in the projective plane}, Math.\ Res.\ Lett.\ \textbf{1} (1994), no.\ 6. 797--808.
%\bibitem{L}P.\ Lisca, \textit{Symplectic fillings and positive scalar curvature}, Geom.\ Topol.\ \textbf{2} (1998), 103--116.
\bibitem{LM1}P.\ Lisca and G.\ Mati\'{c}, \textit{Tight contact structures and Seiberg-Witten invariants}, Invent.\ Math.\ \textbf{129} (1997) 509--525.
\bibitem{LM2}P.\ Lisca and G.\ Mati\'{c}, \textit{Stein 4-manifolds with boundary and contact structures}, Topology and its Applications \textbf{88} (1998) 55--66.
%\bibitem{LS1}P.\ Lisca and A.\,I.\ Stipsicz, \textit{Ozsv\'{a}th-Szab\'{o} invariants and tight contact three-manifolds.\ I}, Geom.\ Topol.\ \textbf{8} (2004), 925--945. 
%\bibitem{LS2}P.\ Lisca and A.\,I. Stipsicz, \textit{Ozsv\'{a}th-Szab\'{o} invariants and tight contact three-manifolds.\ II}, J.\ Differential Geom.\ \textbf{75} (2007), 109--141.
%\bibitem{LoP}A.\ Loi and R.\ Piergallini, \textit{Compact Stein surfaces with boundary as branched covers of $B^4$}, Invent.\ Math.\ \textbf{143}  (2001),  no.\ 2, 325--348.
%\bibitem{M}R.\ Matveyev, \textit{A decomposition of smooth simply-connected $h$-cobordant $4$-manifolds}, J.\ Differential Geom.\ \textbf{44} (1996), no.\ 3, 571--582.
%\bibitem{Ma}B.\ Mazur, \textit{A note on some contractible $4$-manifolds}, 
%Ann. of Math.\ \textbf{73} (1961), 221--228.
%\bibitem{MMS}J.\ Morgan, T.\ Mrowka and Z.\ Szab\'{o}, \textit{Product formulas along $T^3$ for Seiberg-Witten invariants}, Math.\ Res.\ Lett.\ \textbf{4} (1997), 915--929.
\bibitem{MST}J.\ Morgan, Z.\ Szab\'{o} and C.\ Taubes, \textit{A product formula for the Seiberg-Witten invariants and the generalized Thom conjecture}, J.\ Differential Geom.\ \textbf{44} (1996), 706--788.
\bibitem{OS1}B.\ Ozbagci and A I.\ Stipsicz, \textit{Surgery on contact 3-manifolds and Stein surfaces}, Bolyai Society Mathematical Studies, 13. Springer-Verlag, Berlin; Janos Bolyai Mathematical Society, Budapest, 2004.
\bibitem{OV}B.\ Ozbagci and O.\ van Koert, \textit{Contact open books with exotic pages}, Arch.\ Math.\ (Basel) \textbf{104} (2015),  no.\ 6, 551--560.
%\bibitem{OS2}B.\ Ozbagci and A I.\ Stipsicz, \textit{Contact 3-manifolds with infinitely many Stein fillings}, Proc. Amer. Math. Soc. \textbf{132} (2004), no. 5, 1549--1558.
\bibitem{OzSz_ad}P.\ Ozsv\'{a}th and Z.\ Szab\'{o}, \textit{The symplectic Thom conjecture}, Ann.\ of Math.\ \textbf{151} (2000), 93--124. 

 
%\bibitem{OzSz}P.\ Ozsv\'{a}th and Z.\ Szab\'{o}, \textit{The symplectic Thom conjecture}, Ann.\ of Math.\ \textbf{151} (2000), 93--124. 
\bibitem{vK}O.\ van Koert, \textit{Open books on contact five-manifolds}, Ann.\ Inst.\ Fourier (Grenoble) \textbf{58} (2008),  no.\ 1, 139--157.
%\bibitem{We}C.\ Wendl, \textit{Strongly fillable contact manifolds and $J$-holomorphic foliations}, Duke Math.\ J.\ \textbf{151} (2010), no.\ 3, 337--384. 
%\bibitem{We2}C.\ Wendl, personal communication.
%\bibitem{Wei}A.\ Weinstein, \textit{Contact surgery and symplectic handlebodies}, Hokkaido Math.\ J.\ \textbf{20} (1991), 241--251.
\bibitem{Y5}K.\ Yasui, \textit{Nuclei and exotic 4-manifolds}, arXiv:1111.0620v2.
\bibitem{Y6}K.\ Yasui, \textit{Partial twists and exotic Stein fillings}, arXiv:1406.0050. 
\end{thebibliography}
\end{document}